\documentclass{amsart}[12pt,reqno]
\makeatletter
\def\section{\@startsection{section}{1}%
	\z@{.7\linespacing\@plus\linespacing}{.5\linespacing}%
	{\bfseries
		\centering
}}
\def\@secnumfont{\bfseries}
\makeatother
\usepackage{mathrsfs, amsmath,amssymb}
\usepackage{fullpage}
\usepackage{mathtools}
 \usepackage{xcolor}
\newtheorem{teo}{Theorem}
\newtheorem{pro}{Proposition}
\newtheorem{lem}{Lemma}
\newtheorem{cor}{Corollary}
\newtheorem*{rem}{Remark}

\title{Summing free unitary Brownian motions with applications to quantum information}

\author{Nizar Demni}
\address{ Aix-Marseille Universit\'e CNRS Centrale
Marseille I2M - UMR 7373. 39 rue F. Joliot Curie, 13453 Marseille, France }
\email{nizar.demni@univ-amu.fr}

\author[T. Hamdi]{Tarek Hamdi}
\address{Department of Management Information Systems \\ College of Business Management \\ Qassim University  \\ Saudi Arabia
and Laboratoire d'Analyse Math\'ematiques et applications LR11ES11 \\ Universit\'e de Tunis El-Manar \\ Tunisie}
\email{ t.hamdi@qu.edu.sa } 


\begin{document}
\maketitle
\begin{abstract} 
Motivated by quantum information theory, we introduce a dynamical random state built out of the sum of $k \geq 2$ independent unitary Brownian motions. 
In the large size limit, its spectral distribution equals, up to a normalising factor, that of the free Jacobi process associated with a single self-adjoint projection with trace $1/k$. Using free stochastic calculus, we extend this equality to the radial part of the free average of $k$ free unitary Brownian motions and to the free Jacobi process associated with two self-adjoint projections with trace $1/k$, provided the initial distributions coincide. 
In the single projection case, we derive a binomial-type expansion of the moments of the free Jacobi process which extends to any $k \geq 3$ the one derived in \cite {DHH} in the special case $k=2$. 
Doing so give rise to a non normal (except for $k=2$) operator arising from the splitting of a self-adjoint projection into the convex sum of $k$ unitary operators. 
This binomial expansion is then used to derive a pde satisfied by the moment generating function of this non normal operator and for which we determine the corresponding characteristic curves.
\end{abstract} 
\tableofcontents
\section{Introduction and Motivation} 
\subsection{Random matrices in quantum information theory}
Randomness lies at the heart of Shannon's pioneering work on classical information theory (see the expository paper \cite{Ver}). It also plays a key role in quantum information theory through the use of techniques from random matrix theory. Actually, the latter  open the way to choose typical random subspaces in large-size quantum systems which violate additivity conjectures for minimum output R\'enyi and von Neumann entropies (see \cite{Col-Nec} and references therein). Here, typicality is taken with respect to the uniform measure in the compact complex Grassmann manifold or equivalently with respect to the Haar distribution in the group of unitary matrices. Note that this distribution together with Ginibre random matrices also served in \cite{CNPZ} to generate random density matrices induced from states in bipartite systems (see \cite{KNPPZ} for similar constructions of quantum channels). 

A natural dynamical version of the Haar distribution in the group of unitary matrices is the so-called unitary Brownian motion (\cite{Lia}). This stochastic process was used in \cite{Nec-Pel} where the authors introduced and studied a random state drawn from the Brownian motion on the complex projective space (the row vector of a unitary Brownian motion). There, the main problem was to write explicitly the joint distributions of tuples formed by the moduli of the state coordinates. This problem was entirely solved in \cite{Dem0} using spherical harmonics in the unitary group. To the best of our knowledge, \cite{Nec-Pel} and \cite{Dem0} are the only papers where the unitary Brownian motion is used as a random model in quantum information theory, in contrast to the high occurrence of Haar-distributed unitary matrices (\cite{Col-Nec}). Moreover, it is tempting and challenging as well to prove finite-time analogues of important results in quantum information theory proved using Haar unitary matrices and their Weingarten Calculus (as summarized in \cite{Col-Nec}). 

In this paper, we appeal once more to the unitary Brownian motion in order to introduce a stochastic process valued in the space of density matrices (see \eqref{Model} below). The large time limit of this process was already constructed in \cite{CNPZ} by 
partially tracing a pure random state in a bipartite quantum system. 

\subsection{The dynamical density matrix}
Let $N \geq 1$ be a positive integer and consider a bipartite quantum system $\mathscr{H}_A \otimes \mathscr{H}_B$,
where $\mathscr{H}_A, \mathscr{H_B},$ are complex $N$-dimensional Hilbert spaces. If $(e_j^A)_{j=1}^N, (e_j^B)_{j=1}^N$, are the canonical basis of $\mathscr{H}_A$ and $\mathscr{H_B}$ respectively, then  
\begin{equation*}
\psi:= \frac{1}{\sqrt{N}}\sum_{j=1}^N e_j^A \otimes e_j^B
\end{equation*}
is referred to as the Bell or maximally-entangled state. Now, consider $k \geq 2$ Haar-distributed unitary matrices $U_{\infty}^1(N), \dots U_{\infty}^k(N),$ and define the vector $\psi^k \in \mathscr{H}_A \otimes \mathscr{H}_B$ by: 
\begin{equation*}
\psi^k(N):= \frac{1}{\sqrt{N}} \sum_{m=1}^k\sum_{j=1}^N\left(U_{\infty}^m(N)e_j^A\right) \otimes e_j^B =  \frac{1}{\sqrt{N}}\sum_{j=1}^N\left( \sum_{m=1}^kU_{\infty}^m(N)e_j^A\right) \otimes e_j^B .
\end{equation*} 
Then the partial trace with respect to $\mathscr{H}_B$ of the pure state associated with $\psi^k(N)$  yields the following reduced state: 
\begin{equation*}
\widetilde{W}_{\infty}^k(N):= \frac{(U_{\infty}^1(N)+\dots+U_{\infty}^k(N))(U_{\infty}^1(N)+ \dots + U_{\infty}^k(N))^{\star}}{\textrm{tr}[(U_{\infty}^1(N)+\dots+U_{\infty}^k(N))(U_{\infty}^1(N)+ \dots + U_{\infty}^k(N))^{\star}]}
\end{equation*}
where $\textrm{tr}$ is the trace operator on the space of $N \times N$ matrices. Since the Haar distribution is the stationary distribution of the unitary Brownian motion, it is then natural to introduce the following stochastic process valued in the space of density matrices: 
\begin{equation}\label{Model}
t \mapsto \widetilde{W}_t^k(N):=  \frac{(U_t^1(N)+\dots+U_t^k(N))(U_t^1(N)+ \dots + U_t^k(N))^{\star}}{\textrm{tr}[ (U_t^1(N)+\dots+U_t^k(N))(U_t^1(N)+ \dots + U_t^k(N))^{\star}]},
\end{equation}
where $(U_{t}^j(N))_{t \geq 0}, 1 \leq j \leq k,$ are $k$ independent unitary Brownian motions. As we shall now explain, introducing this model is not simply a matter of replacing Haar-distributed matrices by unitary Brownian motions. 
Indeed, the large-size limit of $\widetilde{W}_t^2(N)$ for fixed time $t$ bears a close connection to an instance of the so-called free Jacobi process. 

\subsection{The large size limit of $\widetilde{W}_{\infty}^k(N)$ and the free Jacobi process}
Recall that independent random matrices behave in the large-size limit, under additional law-invariance assumptions, as $\star$-free operators (in Voiculescu's sense) in a tracial non commutative probability space, say $(\mathscr{A}, \tau)$ (\cite{Min-Spe}). 
For instance, independent Haar-distributed unitary matrices converge strongly and almost surely as $N \rightarrow \infty$ to Haar-distributed unitary operators (see \cite{CDK} and references therein). 
Consequently, the operator norm of $\widetilde{W}_{\infty}^k(N)$ converges almost surely as $N \rightarrow \infty$ to 
\begin{equation*}
\widetilde{W}_{\infty}^k := \frac{(U_{\infty}^1+\dots+U_{\infty}^k)(U_{\infty}^1+ \dots + U_{\infty}^k)^{\star}}{\tau[(U_{\infty}^1+\dots+U_{\infty}^k)(U_{\infty}^1+ \dots + U_{\infty}^k)^{\star}]}, 
\end{equation*}
where $\{U_{\infty}^j, 1 \leq j \leq k\}$ is a $k$-tuple of Haar unitary operators which are $\star$-free in $(\mathscr{A}, \tau)$. Note that 
\begin{equation*}
\widetilde{W}_{\infty}^k := \frac{(U_{\infty}^1+\dots+U_{\infty}^k)(U_{\infty}^1+ \dots + U_{\infty}^k)^{\star}}{k}, 
\end{equation*}
since $\tau(U_{\infty}^j) = 0$ and since $\tau(U_{\infty}^j(U_{\infty}^m)^{\star}) = \tau(U_{\infty}^j)\tau((U_{\infty}^m)^{\star}) = 0$ for any $1 \leq j < m \leq k$. In particular, when $k=2$, the invariance of the Haar distribution shows further that 
$\widetilde{W}_{\infty}^2/2$ is equally distributed as:
\begin{equation*}
\frac{({\bf 1} + U_{\infty}^1)({\bf 1}+U_{\infty}^1)^{\star}}{4} = \frac{2{\bf 1} + U_{\infty}^1+(U_{\infty}^1)^{\star}}{4}.
\end{equation*}
The spectral distribution of this Hermitian operator is known to be the arcsine distribution (\cite{HP}). It also coincides with an instance of the stationary distribution of the so-called free Jacobi process (\cite{Dem}). As shown in \cite{DHH}, 
this coincidence actually holds at any time $t > 0$: the spectral distribution of the free Jacobi process (in the compressed algebra) associated with an orthogonal projection of rank $1/2$ is the same as that of  
\begin{equation}\label{Op1}
\frac{2{\bf 1} + U_{2t} + U_{2t}^{\star}}{4} \quad \in \mathscr{A},
\end{equation}
where now $(U_s)_{s \geq 0}$ is a free unitary Brownian motion. As proved by Biane (\cite{Bia}), when properly time-rescaled, $(U_s)_{s \geq 0}$ is the large-size limit of the unitary Brownian motion  (\cite{Bia}). It is a unitary free L\'evy process as well with respect to the free multiplicative convolution on the unit circle (\cite{HP}). In particular, the spectral distribution of \eqref{Op1} coincides with that of  
\begin{equation}\label{Op2}
\frac{(U_t^1 + U_t^2)(U_t^1 + U_t^2)^{\star}}{4}
\end{equation}
where $(U_s^1)_{s \geq 0}, (U_s^2)_{s \geq 0}$ are two free copies of $(U_s)_{s \geq 0}$. Besides, \eqref{Op2} is the large-size limit $N \rightarrow \infty$ of (properly time-rescaled)
\begin{equation*}
\frac{(U_{t}^1(N)+U_{t}^2(N))(U_{t}^1(N)+U_{t}^2(N))^{\star}}{4}.
\end{equation*}
Up to a normalization, this Hermitian random matrix is nothing else but $(1/2)\widetilde{W}_t^2(N)$ which converges almost surely and strongly to 
\begin{equation}\label{Op2}
\frac{(U_t^1 + U_t^2)(U_t^1 + U_t^2)^{\star}}{2[2+\tau(U_t^1(U_t^2)^{\star})+ \tau(U_t^2(U_t^1)^{\star})]} = \frac{(U_t^1 + U_t^2)(U_t^1 + U_t^2)^{\star}}{4(1+e^{-t})},
\end{equation}
where the second equality follows from $\tau(U_t^1(U_t^2)^{\star}) = \tau(U_t^1)\tau((U_t^2)^{\star}) = e^{-t}$ (\cite{Bia}). 
In a nutshell, the free Jacobi process associated with an orthogonal projection with trace $1/2$ is, up to a normalising factor, the large-size limit of (the properly rescaled) $\widetilde{W}_t^2(N)$.

\subsection{Main results}
 The above picture extends to any integer $k \geq 2$ as follows. On the one hand, $\widetilde{W}_t^k(N)$ converges strongly and almost surely (properly time-rescaled), to the self-adjoint and unit-trace operator (\cite{CDK}): 
\begin{equation*}
\widetilde{W}_t^k := \frac{(U_t^1+\dots+U_t^k)(U_t^1+ \dots + U_t^k)^{\star}}{\tau[(U_t^1+\dots+U_t^k)(U_t^1+ \dots + U_t^k)^{\star}]} = \frac{(U_t^1+\dots+U_t^k)(U_t^1+ \dots + U_t^k)^{\star}}{k[1+(k-1)e^{-t}]},
\end{equation*}
where $(U_s^j)_{s \geq 0}, 1 \leq j \leq k$ are free copies of $(U_s)_{s \geq 0}$ in $(\mathscr{A}, \tau)$. On the other hand, if 
\begin{equation*}
G_t^k := U_t^1+\dots+U_t^k, \quad t \geq 0,
\end{equation*}
then Nica and Speicher's boxed convolution (\cite{Nic-Spe96}) implies that the $\star$-moments of $G_t^k$ in $(\mathscr{A}, \tau)$ coincide with those of $PU_tP$ in the compressed space $(P\mathscr{A}P, k\tau)$, where $P \in \mathscr{A}$ is a selfadjoint projection which is free from $(U_s, U_s^{\star})_{s \geq 0}$ and has rank $\tau(P) = 1/k$. Consequently, the corresponding Brown measures coincide and so do the spectral distributions of their radial parts, namely
\begin{equation*}
\frac{W_t^k}{k^2} := \frac{G_t^k (G_t^k )^\star}{k^2}=\frac{(U_t^1+\dots+U_t^k)}{k}\frac{(U_t^1+ \dots + U_t^k)^{\star}}{k}, 
\end{equation*}
in $(\mathscr{A}, \tau)$ and
\begin{equation*}
(PU_tP)(PU_tP)^{\star} = PU_tPU_t^{\star}P
\end{equation*}
 in $(P\mathscr{A}P, k\tau)$. 
In particular, the reduced density matrix $\widetilde{W}_t^k(N)$ and its large-size limit $\widetilde{W}_t^k$ complete the following commutative diagram: 
\begin{equation*}
\begin{array}{lcr}
\widetilde{W}_t^k(N) & t \longrightarrow \infty & \widetilde{W}_{\infty}^k(N) \\ 
& & \\ 
\substack{N \rightarrow \infty} \Big\downarrow & & \Big\downarrow \substack{N \rightarrow \infty} \\
& & \\
\widetilde{W}_t^k &  t \longrightarrow \infty &  \widetilde{W}_{\infty}^k
\end{array}.
\end{equation*}
Now, let $Q \in \mathscr{A}$ be another selfadjoint projection and assume $Q$ is $\star$-free with $(U_t)_{t \geq 0}$. Then, we may consider more generally the operator $PU_tQ$ and its radial part $PU_tQU_t^{\star}P$. Viewed as an operator in the compressed probability space, the latter defines the free Jacobi process associated with $(P,Q)$. Using free stochastic calculus, we shall prove that for any $t > 0$, the moment sequences of $W_t^k/k^2$ in $(\mathscr{A}, \tau)$ and of $PU_tQU_t^{\star}P$ in 
$(P\mathscr{A}P, k\tau)$ satisfy the same recurrence relation provided that the self-adjoint projections $P$ and $Q$ have common rank $1/k$. Of course, the initial values at $t=0$ of these moment sequences may be different in which case their corresponding spectral distributions will be different as well. On the other hand, it is straightforward to see that the moments of $W_t^k/k^2$ converge as $k \rightarrow \infty$ to $(e^{-nt})_{n \geq 0}$ for fixed time $t$, which contrasts the weak convergence of 
$W_{\infty}^k/k$ to the Marchenko-Pastur distribution. This contrast may be explained by the complicated structure of the $\star$-cumulants of $U_t$ in comparison with those of $U_{\infty}$ (\cite{DGPN}). 

Back to the case $P=Q$ (for sake of simplicity), the relation between the spectral distributions of $W_t^k/k^2$ and of $PU_tPU_t^{\star}P$ opens the way to compute the moments of the former by studying the latter. Indeed, for any $n \geq 1$, $\tau[(W_t^k)^n]$
is a linear combination of $k^{2n}$ factors of the form 
\begin{equation*}
\tau[U_t^{i_1}(U_t^{i_2})^{\star} \cdots (U_t^{i_{2m-1}})^{\star}\ U_t^{i_{2m}}], \quad 1 \leq m \leq n, \quad i_j \in \{1,\dots, k\}.
\end{equation*}
Apart from constant ones, those where any index $i_j$ occurs at most once may be computed using the L\'evy property of the free unitary Brownian motions. However, the contributions of the remaining factors may be only computed using the freeness property 
and as such, the complexity of $\tau[(W_t^k)^n]$ increase rapidly even for small orders. For that reason, we rather focus on the moments of $PU_tPU_t^{\star}P$ and our main result (Theorem \ref{Th2} below) establishes for any $n \geq 1$ a binomial-type expansion of 
\begin{equation*}
k\tau[(PU_tPU_t^{\star}P)^n]
\end{equation*}
as a linear combination of the moments
\begin{equation*}
\tau[(T_kU_tT_kU_t^{\star})^j], \quad 0 \leq j \leq n,
\end{equation*} 
where $T_k := kP - {\bf 1} = T_k^{\star}$ satisfies $\tau(T_k) = 0$. This expansion extends to any integer $k \geq 3$ the expansion proved in \cite{DHH} for $k=2$ for which $T_2 = 2P-1$ is unitary and self-adjoint, which in turn implies that $T_2U_tT_2U_t^{\star}$ is distributed as $U_{2t}$. However, for any $k \geq 3$, $T_k$ is not even normal: it is the sum of $(k-1)$ unitary operators and satisfies the relation 
\begin{equation*}
\tau[(T_k)^2]= (k-2)T_k + k-1. 
\end{equation*} 
Of course, the constant term of this expansion is nothing else but the $n$-th moment of the spectral distribution of $PU_{\infty}PU_{\infty}^{\star}P$ which may be written as a weighted sum of Catalan numbers.  Surprisingly, the higher order coefficients are still given by the binomial coefficients 
\begin{equation*}
\binom{2n}{n-j}, \quad 1 \leq j \leq n,
\end{equation*}
up to the multiplicative factors $(k-1)^{n-j}, 1 \leq j \leq n$. Our proof of the binomial-type expansion is enumerative and technical and it would be interesting to seek a combinatorial proof explaining the occurrence of the binomial coefficients above which form the so-called the Catalan triangle (\cite{Sha}). 

Once the binomial expansion derived, we turn it into a relation between the moment generating functions of $J_t$ and of $T_kU_tT_kU_t^{\star}$. Using the partial differential equation (hereafter pde) satisfied by the former (\cite{Dem}), we derive a pde for the latter. The characteristic curves of this pde seems out of reach for the time present and we postpone their analysis to a future research work. 

The paper is organised as follows. In the next section, we discuss the relation between the $\star$-moments of the free average of $k$ free unitary Brownian motion and those of the compression $PU_tP$ when $\tau(P) = 1/k$. There, we also prove that the moment sequences of the radial parts of the free average and of $PU_tQ$ satisfy the same recurrence relation and that their limits as $k \rightarrow \infty$ is the Dirac mass at $e^{-t}$. In the third section, we prove the binomial-type formula for the moments of $J_t$ then turn it into a relation between moment generating functions whence we deduce a pde for the moment generating function of $T_kU_tT_kU_t^{\star}$. We also include two appendices where we prove two formulas which we could not find in literature and which we think are of independent interest. The first formula has the merit to express the moments of the stationary distribution of the free Jacobi process corresponding to $\tau(P) = 1/k$ as a perturbation of those corresponding to $\tau(P) = 1/2$. In particular, it involves a family of polynomials with integer coefficients in the variable $(k-2)$ and its derivation relies on special properties of the Gauss hypergeometric function. As to the second formula, it concerns the free cumulants of a self-adjoint projection with arbitrary rank which we express as a difference of two Legendre polynomials.

\section{Relating $G^t_k$ and compressions of $U_t$}
\subsection{Warm up: $\star$-moments of compressions}. Given a collection of operators $(a_1, \dots, a_n)$ in a non commutative probability space $(\mathscr{A}, \tau)$, their joint distribution $\mu_{a_1, \dots, a_n}$ is the linear functional which assigns to any polynomial $P$ in $n$ non commuting indeterminates its trace $\tau(P(a_1, \dots, a_n))$. In this respect, the Nica-Speicher generalized $R$-transform (\cite{Nic-Spe96}) allows to relate the joint distribution of the compressed collection $(Pa_1P, \dots, Pa_nP)$ by a free self-adjoint projection $P$ in the compressed algebra to $\mu_{a_1, \dots, a_n}$. In particular, when $n=2$ and if $a_1 = a, a_2 = a^{\star}$ then $\mu_{a, a^{\star}}$ is given by all the $\star$-moments of $a$ and we shall simply refer to it as the distribution of $a$. The following result shows that if 
$\tau(P) = 1/k, k \geq 2,$ then the compression of $(U_t, U_t^{\star})$ by $P$ amounts to summing $k$ free copies of $(U_t, U_t^{\star})$ up to dilation. Though we expect that this result is known among the free probability community, we did not find it written anywhere and we include it here for the reader's convenience. Note also that it reduces to the Nica-Speicher convolution semi-group when $a$ is self-adjoint.  
\begin{pro}\label{Pr1}
	Let $P$ be a self-adjoint projection freely independent from $\{U_t,U_t^\star\}_{t \geq 0}$ with $\tau(P) = 1/k, k \geq 2$. 
	Then, the distribution of $PU_tP$ in $(P\mathscr{A}P, k\tau)$ coincides with that of $G_t^k/k$ in $(\mathscr{A}, \tau)$.  
\end{pro}
\begin{proof}
	Given an operator $a \in \mathscr{A}$, let $R(\mu_{a,a^{\star}})$ be its generalized $R$-transform (\cite{Nic-Spe96}, section 3.9) and recall that it entirely determines the distribution of $a$. Then, one has on the one hand:
	\begin{equation*}
		R\left(\mu_{\frac{1}{k}U_t^1+\ldots+\frac{1}{k}U_t^k,(\frac{1}{k}U_t^1)^\star+\ldots+(\frac{1}{k}U_t^k)^\star}\right)=kR\left(\mu_{\frac{1}{k}U_t,\frac{1}{k}U_t^\star}\right)
	\end{equation*}
	due to the $\star$-freeness of $(U_t^j)_{j=1}^k$ (\cite{Nic-Spe96}).
On the other hand, \cite[Application 1.11]{Nic-Spe96} entails
	\begin{equation*}
		R\left(\mu_{PU_tP,PU_t^\star P}\right)= k R\left(\mu_{\frac{1}{k}U_t,\frac{1}{k}U_t^\star}\right),
	\end{equation*}
	where the distribution $\mu_{PU_tP,PU_t^\star P}$ is considered in the compressed space $(P\mathscr{A}P, k\tau)$.
\end{proof}	
The Brown measure of a non normal operator plays a key role in random matrix theory since it supplies a candidate for the limiting empirical distribution of a non normal matrix. 
In a tracial non commutative probability space, it is fully determined by $\star$-moments and one immediately deduces from the previous proposition that the Brown measures of $G_t^k/k$ and of $PU_tP$ coincide when $\tau(P) = 1/k$. In general, 
the description of the Brown measure of $PU_tP$ is a quite difficult problem: the main result proved in \cite{Dem-Ham1} already provides a Jordan domain containing its support. As a matter of fact, Proposition \ref{Pr1} offers another way to compute the Brown measure of $PU_tP$ in the particular case $\tau(P) = 1/k$ relying on operator-valued free probability as explained in \cite{BSS}. However, it turns out that the computations are already tedious even for $k=2$ and as such, we postpone them to a future research work. 

In the long-time regime $t \rightarrow \infty$, the fact that the $R$-diagonal operator $PU_{\infty}P$ and the average of $k$ free Haar unitaries share the same Brown measure is transparent from Haagerup-Laarsen results (\cite{Haa-Lar}, examples 5.3 and 5.5) though not being explicitly pointed out there. Indeed, this measure is radial and absolutely continuous with density given by (\cite{Haa-Lar}): 
\begin{equation*}
\frac{k-1}{\pi(1-|\lambda|^2)^2}{\bf 1}_{(0,1/\sqrt{k})}(|\lambda|) d\lambda,
	\end{equation*}
with respect to Lebesgue measure $d\lambda$. 

\subsection{Radial parts and beyond} \label{radial}
If we consider the radial parts of $PU_tP$ and of $G_t^k/k$, then the equality between the moments of $PU_tPU_t^{\star}P$ and of $W_t^k/k^2$ may be readily deduced from the moment-cumulant formula  for the compression by a free projection (see Theorem 14.10, \cite{Nic-Spe}). 
Actually, if $(a_1, \dots, a_m)$ is collection of operators in $\mathscr{A}$ which is free from $P$, then 
\begin{equation}\label{Comp}
\frac{1}{\tau(P)}\tau(Pa_{i_1}Pa_{i_2}P\dots Pa_{i_n}P) = \sum_{\pi \in NC(n)} \kappa_{\pi}[a_{i_1}, \dots, a_{i_n}] [\tau(P)]^{n-|\pi|},
\end{equation} 
for any indices $1 \leq i_1, \dots, i_n, \leq m$. Here $NC(n)$ is the lattice of non crossing partitions, $|\pi|$ is the number of blocks of the partition $\pi \in NC(n)$ and $\kappa_{\pi}$ is the multiplicative functional of free cumulants of blocks of $\pi$ 
(see Lectures 10 and 11 in \cite{Nic-Spe} for more details). Specializing \eqref{Comp} with $(a_{i_{2j+1}}, a_{i_{2j+2}}) = (U_t, U_t^{\star}), 0 \leq j \leq n-1,$ and $\tau(P) = 1/k$, we get: 
\begin{equation}\label{Comp1}
 k\tau(PU_tPU_t^{\star}P\dots PU_tPU_t^{\star}P) = \sum_{\pi \in NC(2n)} \kappa_{\pi}[\underbrace{U_t, U_t^{\star}\dots, U_t, U_t^{\star}}_{2n}] k^{|\pi|-2n}. 
\end{equation} 
On the other hand, the moment-cumulant formula (11.8) in \cite{Nic-Spe} entails: 
\begin{equation}\label{Comp2}
\frac{1}{k^{2n}}\tau[(W_t^k)^n] = \frac{1}{k^{2n}}\sum_{\pi \in NC(2n)} k_{\pi}[G_t, (G_t^k)^{\star}, \dots, G_t, (G_t^k)^{\star}]
\end{equation}
where we recall that $G_t^k = U_t^1+\dots+U_t^k$ and $W_t^k = G_t^k(G_t^k)^{\star}$. But if $V$ is a block of $\pi$ then $\kappa_V$ is the sum of terms of the form 
\begin{equation*}
\kappa_{|V|}[(U_t^{j_1})^{\epsilon(1)}, (U_t^{j_2})^{\epsilon(2)}, \dots, (U_t^{j_{2n}})^{\epsilon(2n)}], 
\end{equation*}
where $\epsilon(1), \dots, \epsilon(2n) \in \{1, \star\}$ and $1 \leq j_1, \dots, j_{2n} \leq k$. All these terms vanish due to the $\star$-freeness of $(U_t^j)_{j=1}^k$ except those of the form 
\begin{equation*}
\kappa_{|V|}[(U_t^{j})^{\epsilon(1)}, (U_t^{j})^{\epsilon(2)}, \dots, (U_t^j)^{\epsilon(2n)}], 
\end{equation*}
 for a single index $1 \leq j \leq k$. There are $k$ such terms and all give the same contribution 
 \begin{equation*}
\kappa_{|V|}[(U_t)^{\epsilon(1)}, (U_t^{j})^{\epsilon(2)}, \dots, (U_t)^{\epsilon(2n)}],
\end{equation*}
since $U_t^1, \dots, U_t^k$ have the same spectral distribution as $U_t$. Consequently, the RHS of \eqref{Comp2} and \eqref{Comp1} are equal. 

More generally, we shall prove below that given two orthogonal projections $P$ and $Q$ which are $\star$-free from $(U_t)_{t \geq 0}$, the moments of $PU_tQU_t^{\star}P$ and those of $W_t^k/k^2$ coincide  provided that $\tau(P) = \tau(Q) = 1/k$. Our main tool is free stochastic calculus and we refer to \cite{Bia-Spe} for further details on this calculus. To proceed, recall from \cite{Bia} the stochastic differential equation satisfied by the free unitary Brownian motion $(U_t)_{t \geq 0}$:
\begin{equation*}
dU_t = iU_t dX_t -\frac{U_t}{2} dt, \quad U_0 = {\bf 1},
\end{equation*}
where $(X_t)_{t \geq 0}$ is a free additive Brownian motion. Hence, there exists a $k$-tuple free additive Brownian motions $(X_t^j)_{t \geq 0}, 1 \leq j \leq k,$ which are free in $\mathscr{A}$ and such that 
\begin{equation}\label{FSDE}
dU_t^j = iU_t^j dX_t^j -\frac{U_t^j}{2} dt, \quad U_0^j = {\bf 1}.
\end{equation}
With the help of the free It\^o formula (\cite{Bia}), we shall prove: 
\begin{teo}
For any $n \geq 1, t > 0$, set\footnote{We omit the dependence on $k$ for sake of clarity.}
\begin{equation*}
s_n(t) := \tau[(W_t^k)^n]. 
\end{equation*}
Then, 
\begin{equation}\label{Moments1}
\partial_t s_n(t) = -ns_n(t) +nks_{n-1}(t) +nk\sum_{j=0}^{n-2}s_{n-j-1}(t)s_j(t) - \frac{n}{k}\sum_{j=0}^{n-2}s_{n-j-1}(t)s_{j+1}(t),
\end{equation}
where an empty sum is zero. 
\end{teo}
\begin{proof}
Using \eqref{FSDE}, we get
\begin{equation*}
dG_t^k = i\sum_{j=1}^kU_t^jdX_t^j - \frac{G_t^k}{2} dt
\end{equation*}
whence 
\begin{align*}
dW_t^k &= d[G_t^k(G_t^k)^{\star}] 
 = dG_t^k (G_t^k)^{\star} + G_t^k (dG_t^k)^{\star} + (dG_t^k)((dG_t^k)^{\star}),
\end{align*}
where $(dG_t^k)((dG_t^k)^{\star})$ is the bracket of the semimartingales $dG_t^k$ and $(dG_t^k)^{\star}$. Since $(X_t^j)_{t \geq 0}$ are assumed free then 
\begin{equation*} 
(dX_t^j)(dX_t^m) = \delta_{jm} dt, \quad 1 \leq j,m \leq k,
\end{equation*}
so that 
\begin{equation*}
dW_t^k = \sum_{j=1}^k\left[(iU_t^j)dX_t^j(G_t^k)^{\star} + G_tdX_t^j(iU_t^j)^{\star}\right] + (k-W_t^k) dt. 
\end{equation*}
Now, borrowing the terminology and the notations of \cite{Bia-Spe}, we introduce the bi-processes: 
\begin{equation*}
F_t^j:= (iU_t^j) \otimes (G_t^k)^{\star} + (G_t^k)\otimes (iU_t^j)^{\star}, \quad 1 \leq j \leq k,
\end{equation*}
and write: 
\begin{equation*}
dW_t^k = \sum_{j=1}^kF_t^j \sharp dX_t^j + (k-W_t^k) dt. 
\end{equation*}
Consequently, for any $n \geq 1$, Proposition 4.3.2 in \cite{Bia-Spe} entails: 
\begin{align*}
d[(W_t^k)^n] &= \textrm{Martingale part} + \sum_{j=0}^{n-1}(W_t^k)^j \otimes (W_t^k)^{n-1-j} \sharp (k-W_t^k) dt  
\\& - \sum_{j=1}^k\sum_{\substack{m,l \geq 0 \\ m+l \leq n-2}}(W_t^k)^{l}U_t^j(G_t^k)^{\star}(W_t^k)^{n-m-l-2} \tau[(W_t^k)^{m}U_t^j(G_t^k)^{\star}]dt 
\\& -  \sum_{j=1}^k\sum_{\substack{m,l \geq 0 \\ m+l \leq n-2}}(W_t^k)^{l}G_t^k(U_t^j)^{\star}(W_t^k)^{n-m-l-2} \tau[(W_t^k)^{m}G_t^k(U_t^j)^{\star}]dt 
\\&+  \sum_{j=1}^k\sum_{\substack{m,l \geq 0 \\ m+l \leq n-2}}\left\{(W_t^k)^{n-m-2} \tau[(W_t^k)^{m+1}] + (W_t^k)^{n-m-1} \tau[(W_t^k)^{m}]\right\}.
\end{align*}
Taking the expectation with respect to $\tau$ of both sides and differentiating with respect to the variable $t$\footnote{The state $\tau$ is tracial and all the processes are continuous in the strong topology.}, we get: 
\begin{align*}
\partial_t s_n(t) &= -ns_n(t) + nk s_{n-1}(t) 
\\& - \sum_{j=1}^k\sum_{\substack{m,l \geq 0 \\ m+l \leq n-2}}\tau[(W_t^k)^{n-m-2}U_t^j(G_t^k)^{\star}] \tau[(W_t^k)^{m}U_t^j(G_t^k)^{\star}]
\\& -  \sum_{j=1}^k\sum_{\substack{m,l \geq 0 \\ m+l \leq n-2}}\tau[(W_t^k)^{n-m-2}G_t^k(U_t^j)^{\star}] \tau[(W_t^k)^{m}G_t^k(U_t^j)^{\star}]
\\&+  \sum_{j=1}^k\sum_{\substack{m,l \geq 0 \\ m+l \leq n-2}}\left\{\tau[(W_t^k)^{n-m-2}] \tau[(W_t^k)^{m+1}] + \tau[(W_t^k)^{n-m-1}] \tau[(W_t^k)^{m}]\right\}. 
\end{align*}
The last (triple) sum yields the following contribution (the summands there do not depend on the indices $j,l$): 
\begin{multline*}
k\sum_{m=0}^{n-2}(n-m-1) \tau[(W_t^k)^{n-m-2}] \tau[(W_t^k)^{m+1}] + k\sum_{m=0}^{n-2}(n-m-1) \tau[(W_t^k)^{n-m-1}] \tau[(W_t^k)^{m}] \\ 
= nk\sum_{m=0}^{n-2}\tau[(W_t^k)^{n-m-1}] \tau[(W_t^k)^{m}],
\end{multline*}
where the last equality follows from the index change $m \mapsto n-m-2$. Finally, the summands 
\begin{equation*}
\tau[(W_t^k)^{n-m-2}U_t^j(G_t^k)^{\star}] \tau[(W_t^k)^{m}U_t^j(G_t^k)^{\star}], \quad 1 \leq j \leq k,
\end{equation*}
do not depend on $j$ since $W_t^k$ and $G_t^k$ are symmetric (invariant under permutations) and since the unitary operators $U_t^j, 1 \leq j \leq k,$ are free and have identical distributions. As a result, 
\begin{align*}
S_1:&=\sum_{j=1}^k\sum_{\substack{m,l \geq 0 \\ m+l \leq n-2}}\tau[(W_t^k)^{n-m-2}U_t^j(G_t^k)^{\star}] \tau[(W_t^k)^{m}U_t^j(G_t^k)^{\star}]dt
\\& = \sum_{j=1}^k\sum_{m=0}^{n-2}(n-m-1)\tau[(W_t^k)^{n-m-2}U_t^j(G_t^k)^{\star}] \tau[(W_t^k)^{m}U_t^j(G_t^k)^{\star}]dt 
\\& = \frac{1}{k}\sum_{m=0}^{n-2}(n-m-1)\sum_{j,l=1}^k )\tau[(W_t^k)^{n-m-2}U_t^j(G_t^k)^{\star}] \tau[(W_t^k)^{m}U_t^l(G_t^k)^{\star}]dt 
\\& = \frac{1}{k}\sum_{m=0}^{n-2}(n-m-1) \tau[(W_t^k)^{n-m-1}] \tau[(W_t^k)^{m+1}].
\end{align*}  
Similarly, 
\begin{align*}
S_2: &= \sum_{j=1}^k\sum_{\substack{m,l \geq 0 \\ m+l \leq n-2}}\tau[(W_t^k)^{n-m-1}G_t^k(U_t^j)^{\star}] \tau[(W_t^k)^{m}G_t^k(U_t^j)^{\star}]dt 
\\& = \frac{1}{k}\sum_{m=0}^{n-2}(n-m-1) \tau[(W_t^k)^{n-m-1}] \tau[(W_t^k)^{m+1}].
\end{align*}
Performing the index change $m \mapsto n-m-2$ in $S_2$, we end up with: 
\begin{equation*}
S_1+S_2 = \frac{n}{k}\sum_{m=0}^{n-2}\tau[(W_t^k)^{n-m-1}] \tau[(W_t^k)^{m+1}].
\end{equation*}
Gathering all the contributions above, we obtain \eqref{Moments1}.
\end{proof}

Setting $r_n(t) := s_n(t)/k^{2n} = \tau[(W_t^k/k^2)^{n}]$, we readily infer from \eqref{Moments1}: 
\begin{cor}
For any $n \geq 1$, 
\begin{equation}\label{Moments2}
\partial_t r_n(t) = -nr_n(t) +\frac{n}{k}r_{n-1}(t) +\frac{n}{k}\sum_{j=0}^{n-2}r_{n-j-1}(t) [r_j(t) - r_{j+1}(t)].
\end{equation}
\end{cor}
The moment relation \eqref{Moments2} is an instance of the one derived in Corollary 6.1 in \cite{Dem}. More precisely, let 
\begin{equation*}
J_t ;= PU_tQU_t^{\star}P
\end{equation*}
be the free Jacobi process associated with the self-adjoint projections $(P,Q)$. Viewed as an operator in the compressed algebra $(P\mathscr{A}P, \tau/\tau(P))$, its moments 
\begin{equation*}
m_n(t) = \frac{\tau(J_t^n)}{\tau(P)}, \quad n \geq 1, \quad m_0(t) = 1, 
\end{equation*}
satisfy the following differential system: 
\begin{equation}\label{Moments3}
\partial_t m_n(t) = -nm_n(t) + n\theta m_{n-1}(t) + n\lambda \theta \sum_{j=0}^{n-2}m_{n-j-1}(t) [m_j(t) - m_{j+1}(t)],
\end{equation}
where $\tau(P) = \lambda \theta \in (0,1], \tau(Q) = \theta \in (0,1]$. Consequently, if $\lambda = 1, \theta = 1/k$ then \eqref{Moments2} and \eqref{Moments3} coincide and in turn both sequences coincide provided that $m_n(0) = r_n(0)$ for all $n \geq 0$. 

\subsection{Limit as $k \rightarrow \infty$}
Let $U_{\infty} \in \mathscr{A}$ be a Haar unitary operator and assume that $U_{\infty}$ is free with $\{P,Q\}$. If $\tau(P) = \tau(Q) = 1/k$ then the spectral distribution of 
\begin{equation*}
J_{\infty}:= PU_{\infty}QU_{\infty}^{\star}P
\end{equation*}
in the compressed algebra $(P\mathscr{A}P, \tau/\tau(P))$ is given by (see e.g. \cite{Dem}, p. 130):
\begin{equation*}
\tilde{\mu}_{\infty}^k(dx) = \frac{1}{2\pi} \frac{\sqrt{4(k-1)x - k^2x^2}}{x(1 -x)}{\bf 1}_{[0,4(k-1)/k^2]}(x) dx, 
\end{equation*}  
Its pushforward under the dilation $x \mapsto kx$ is readily computed as 
\begin{equation*}
\mu_{\infty}^k(dx) = \frac{1}{2\pi} \frac{\sqrt{4k(k-1)x - k^2x^2}}{kx - x^2}{\bf 1}_{[0,4(k-1)/k]}(x) dx, 
\end{equation*}  
and converges weakly to the Marchenko-Pastur distribution of parameter one (\cite{CNPZ}): 
\begin{equation*}
\nu_{MP}(du) := \frac{1}{2\pi}\sqrt{\frac{4-u}{u}}{\bf 1}_{[0,4]}(u) du.  
\end{equation*}
IF $W_{\infty}^k = (U_{\infty}^1+\dots + U_{\infty}^k)(U_{\infty}^1+\dots + U_{\infty}^k)^{\star}$ then we can rephrase the weak convergence above as follows: for any $n \geq 0$
\begin{equation}\label{DLimit}
\lim_{k \rightarrow \infty}\lim_{t \rightarrow \infty} \frac{\tau[(W_t^k)^n]}{k^n} = \lim_{k \rightarrow \infty} \frac{\tau[(W_{\infty}^k)^n]}{k^n} = \int_0^1u^n \nu_{MP}(du) = \frac{4^n(1/2)_n}{(n+1)!}.
\end{equation}
The normalization by $k^n$ may be guessed from the moment-cumulant expansion:
\begin{equation*}
\tau[(W_{\infty}^k)^n] = \sum_{\pi \in NC(2n)} \kappa_{\pi}[\underbrace{U_{\infty}, U_{\infty}^{\star}\dots, U_{\infty}, U_{\infty}^{\star}}_{2n}] k^{|\pi|},
\end{equation*}
since partitions $\pi \in NC(2n)$ with more than $(n+1)$ blocks have zero contribution. Indeed, in any such partition, at least one block admits an odd number of elements in which case the corresponding free $\star$-cumulant vanishes (see \cite{Nic-Spe}, Proposition 15.1). 

For fixed time $t > 0$, the situation becomes different since the free $\star$-cumulants of $U_t$ admit a considerably more complicated structure compared with those of $U_{\infty}$ (\cite{DGPN}). In this respect, we can prove the following limiting result under the stronger normalization $k^2$, which shows that reversing the order of the $(k,t)$ limits in \eqref{DLimit} does not lead to a finite limit.

\begin{pro}
For any $n \geq 0, t \geq 0$: 
\begin{equation*}
\lim_{k \rightarrow \infty} \tau\left(\frac{(W_t^k)^n}{k^{2n}}\right) = \left[\tau(U_t)\right]^{2n} = e^{-nt}.
\end{equation*}
In particular, the free Jacobi process $(PU_tPU_t^{\star}P)_{t \geq 0}$ with $\tau(P) = 1/k$ converges weakly as $k \rightarrow \infty$ to the constant $e^{-t}$ in the compressed algebra.  
\end{pro}
\begin{proof}
From \eqref{Comp1}, we readily see that the limit as $k \rightarrow \infty$ of $k\tau(J_t^n)$ is given by the (non crossing) partition with $2n$ blocks. Therefore, 
 \begin{equation*}
\lim_{k \rightarrow \infty} \tau\left(\frac{(W_t^k)^n}{k^{2n}}\right) = \underbrace{c_1(U_t)c_1(U_t^{\star})\dots c_1(U_t)c_1(U_t^{\star})}_{2n \,\, \textrm{terms}}.
\end{equation*}
Since $c_1(U_t) = c_1(U_t^{\star}) = \tau(U_t) = e^{-t/2}$ (see e.g. \cite{DHH} and references therein), the proposition follows. 
\end{proof}

\section{Analysis of the moments of the free Jacobi process}
\subsection{Moments binomial-type formula}
For sake of simplicity, we restrict our study from now on to the free Jacobi process associated with a single projection $P$. Recall from \cite{DHH} that when $\tau(P) = 1/2$, the moments of the free Jacobi process are linear combinations of those of $U_{2t}$. Indeed, it was observed there that 
\begin{equation}\label{Expk2}
2\tau(J_t^n) = \frac{1}{2^{2n}} \binom{2n}{n} + \frac{2}{2^{2n}} \sum_{j=1}^n \binom{2n}{n-j} \tau[(SU_tSU_t^{\star})^j],
\end{equation}
where $S = 2P-{\bf 1}$ satisfies $S = S^{\star} = S^{-1}$. 
Moreover, Lemme 3.8 in \cite{Haa-Lar} together with the semi-group property of $(U_t)_{t \geq 0}$ show that the spectral distributions of $SU_tSU_t^{\star}$ and of $U_{2t}$ coincide. More generally,   write:
\begin{equation*}
P = \frac{1}{k}\sum_{j=0}^{k-1}S_{j,k}, \quad S_{j,k}:= e^{2i\pi j({\bf 1}-P)/k}. 
\end{equation*}
Then $S_{j,k}$ is a unitary operator satisfying $(S_{j,k})^k = {\bf 1}$ and 
\begin{equation*}
S_{j,k} = {\bf 1} + (\omega_{j,k} - 1)({\bf 1}-P),
\end{equation*} 
where $\omega_{j,k} = e^{2i\pi j/k}$ is the $k$-th root of unity. Set 
\begin{equation*}
T_k := kP-{\bf 1} = \sum_{j=1}^{k-1}S_{j,k}. 
\end{equation*}
Then $T_k$ is self-adjoint and $(T_k)^2 = k(k-2)P + {\bf 1} = (k-2)T_k + (k-1){\bf 1}$. In this respect, we shall prove the following generalization of \eqref{Expk2}: 
\begin{teo}\label{Th2}
For any $k \geq 2$ and any $n \geq 1$, 
\begin{equation*}
m_n(t) = k \tau(J_t) = m_n(\infty) + \frac{k}{k^{2n}}\sum_{j=1}^n (k-1)^{n-j} \binom{2n}{n-j} \tau[(T_kU_tT_kU_t^{\star})^j],
\end{equation*}
where $m_n(\infty)$ is the $n$-th moment of $J_{\infty}$ in $(P\mathscr{A}P, k\tau)$, given by \eqref{MomStat1} and \eqref{MomStat2}. 
\end{teo}
The proof of this Theorem relies on the following four key lemmas. 
\begin{lem}\label{Lem1}
Let $a, b \in \mathscr{A}$ be two operators satisfying $a^2=(k-2)a + (k-1){\bf 1}$, $b^2=(k-2) b+ (k-1){\bf 1}$. Then, the expansion of $[(a+1)(b+1)]^n$ is uniquely written as:
\begin{equation}\label{Expansion1}
[({\bf 1}+a)({\bf 1}+b)]^n=m_n {\bf 1}+ \sum_{j=1}^n c_{n,j}(ab)^j+ \sum_{j=1}^{n-1} d_{n,j}(ba)^j+ \sum_{j=0}^{n-1} e_{n,j}(ab)^ja+ \sum_{j=0}^{n-1} f_{n,j}(ba)^jb
\end{equation}
for some integer sequences $m_n, (c_{n,j}), (d_{n,j}), (e_{n,j}), (f_{n,j})$ satisfying 
\begin{eqnarray*}
m_n &=&  f_{n,0} = e_{n,0}, \\ 
c_{n,j} & = & f_{n,j-1} = e_{n,j-1}, \quad 1 \leq j \leq n, \\  
d_{n,j}  & = & c_{n,j+1}, \quad 1 \leq j \leq n-1. 
\end{eqnarray*}
\end{lem}
\begin{proof}
For sake of clarity, we shall omit the notation ${\bf 1}$ in front of the constant terms. Firstly, the uniqueness follows from the fact that the expansion is a reduced expression.  
Now, since $a(a+{\bf 1}) = (k-1)(a+{\bf 1})$ then 
\begin{align*}
(k-1)[(a+{\bf 1})(b+{\bf 1})]^n  & = a[({\bf 1}+a)({\bf 1}+b)]^n 
\\& = m_n a + \sum_{j=1}^n c_{n,j}a(ab)^j+ \sum_{j=1}^{n-1} d_{n,j}a(ba)^j+ \sum_{j=0}^{n-1} e_{n,j}a(ab)^ja+ \sum_{j=0}^{n-1} f_{n,j}a(ba)^jb
\\& = m_na + (k-2)\sum_{j=1}^nc_{n,j}(ab)^j + (k-1) \sum_{j=0}^{n-1}c_{n,j+1}(ba)^jb + \sum_{j=1}^{n-1}d_{n,j}(ab)^ja 
\\& + \sum_{j=0}^{n-1}e_{n,j}[(k-2)(ab)^ja+(k-1)(ba)^j]  + \sum_{j=1}^{n}f_{n,j-1}(ab)^j. 
\\& =  (k-1)e_{n,0}+ \sum_{j=1}^n[(k-2)c_{n,j}+f_{n,j-1}](ab)^j + (k-1) \sum_{j=1}^{n-1}e_{n,j}(ba)^j + \\&
 + (m_n+(k-2)e_{n,0})a+  \sum_{j=1}^{n-1}[(k-2)e_{n,j}+d_{n,j}](ab)^ja + (k-1) \sum_{j=0}^{n-1}c_{n,j+1}(ba)^jb. 
\end{align*}
Multiplying \eqref{Expansion1} by $(k-1)$ and using the uniqueness of the coefficients, we readily get: 
\begin{equation}\label{Rec01}
m_n = e_{n,0}, \quad c_{n,j} = f_{n,j-1}, 1 \leq j \leq n, \quad e_{n,j} = d_{n,j}, 1 \leq j \leq n-1.
\end{equation}
Similarly, $b(b+{\bf 1}) = (k-1)(b + {\bf 1})$ so that 
\begin{align*}
(k-1)[(a+{\bf 1})(b+{\bf 1})]^n  & = [({\bf 1}+a)({\bf 1}+b)]^n b
\\& = m_n b + \sum_{j=1}^n c_{n,j}(ab)^jb+ \sum_{j=1}^{n-1} d_{n,j}(ba)^jb+ \sum_{j=0}^{n-1} e_{n,j}(ab)^jab+ \sum_{j=0}^{n-1} f_{n,j}(ba)^jb^2
\\& = m_nb + (k-2)\sum_{j=1}^nc_{n,j}(ab)^j + (k-1) \sum_{j=0}^{n-1}c_{n,j+1}(ab)^ja + \sum_{j=1}^{n-1}d_{n,j}(ba)^jb 
\\& + \sum_{j=1}^{n}e_{n,j-1}(ab)^{j}  + (k-2) \sum_{j=0}^{n-1}f_{n,j}(ba)^jb + (k-1)\sum_{j=0}^{n-1}f_{n,j}(ba)^j. 
\\& =  (k-1)f_{n,0}+ \sum_{j=1}^n[(k-2)c_{n,j}+e_{n,j-1}](ab)^j + (k-1) \sum_{j=1}^{n-1}f_{n,j}(ba)^j  \\&
 + (m_n+(k-2)f_{n,0})b+ (k-1) \sum_{j=0}^{n-1}c_{n,j+1}(ab)^ja + \sum_{j=1}^{n-1}[(k-2)f_{n,j}+d_{n,j}](ba)^jb . 
\end{align*}
The uniqueness property again yields: 
\begin{equation}\label{Rec02}
m_n = f_{n,0}, \quad c_{n,j} = e_{n,j-1}, 1 \leq j \leq n, f_{n,j} = d_{n,j}, 1 \leq j \leq n-1.
\end{equation}
Combining \eqref{Rec01} and \eqref{Rec02}, the lemma is proved. 
\end{proof} 
According to Lemma \ref{Lem1}, we only need to focus on the sequences $(m_n)_n, (c_{n,j})_{1 \leq j \leq n}$. The former is closely related to the moment sequence $m_n(\infty)$ of $\tilde{\mu}_{\infty}$.
As to the latter, it satisfies the following relations: 
\begin{lem}\label{Lem2}
For any $2 \leq j \leq n-1$, 
\begin{equation}\label{R01}
c_{n+1,j}= (k-1)c_{n,j}+c_{n,j-1} + (k-1)^2e_{n,j} + (k-1)e_{n,j-1},  
\end{equation}
while 
\begin{equation}\label{Sys01}
\begin{cases}
	c_{n+1,n+1}=c_{n,n}=1 \\
	c_{n+1,n}=c_{n,n-1}+ (k-1) + (k-1)e_{n,n-1} \\ 
	c_{n+1,1}=(k-1)c_{n,1}+(k-1)^2e_{n,1}+(k-1) e_{n,0}+m_n
	\end{cases}.
\end{equation}
\end{lem}
\begin{proof}
Follows readily from 
\begin{equation*}
[({\bf 1}+a)({\bf 1}+b)]^{n+1} = [({\bf 1}+a)({\bf 1}+b)]^n({\bf 1}+a + ab),
\end{equation*}
together with the identities: 
\begin{eqnarray*}
(ab)^j &=& (ab)^{j-1}(ab), \\ 
(ab)^jb & = & (k-2)(ab)^j + (k-1)(ab)^{j-1}a, \\ 
((ab)^{j-1}a)b & = &  (ab)^j, \\ 
((ab)^j a)a & = & (k-2)(ab)^j a + (k-1)(ab)^j, \\ 
((ab)^ja)ab & = & (k-2)(ab)^{j+1}+ (k-1) (ab)^j b \\ 
		 & = & (k-2)(ab)^{j+1}+ (k-1)(k-2) (ab)^j + (k-1)^2(ab)^{j-1}a.
\end{eqnarray*}
\end{proof}
Note that Lemma \eqref{Lem1} allows to rewrite \eqref{R01} and \eqref{Sys01} as 
\begin{equation}\label{Sys02}
\begin{cases}
	c_{n+1,j}= 2(k-1)c_{n,j}+c_{n,j-1} + (k-1)^2c_{n,j+1}, \quad 2 \leq j \leq n+1,\\
	c_{n+1,1}= (2k-1)c_{n,1}+(k-1)^2c_{n,2},
	\end{cases}
\end{equation}
where we set $c_{n,j} = 0, j > n$. Next, we need the following routine computations to prove Lemma \ref{Lem4} below and which give our first formula for $m_n(\infty)$: 

\begin{lem}
	 	For any $n\ge 1$, we have
	 \begin{equation}
	 	m_n(\infty)-m_{n+1}(\infty)=  \frac{(k-1)^{n+1}}{k^{2n+1}} C_n,
	 \end{equation}
 where $C_n$ is the $n$-th Catalan number. In particular, 
 \begin{equation}\label{MomStat1}
m_n(\infty) = 1-\sum_{j=0}^{n-1}\frac{(k-1)^{j+1}}{k^{2j+1}} C_j. 
 \end{equation}
\end{lem}
\begin{proof}
\begin{align*}
	m_n(\infty) - m_{n+1}(\infty) &= \frac{1}{2\pi}  \int x^{n-1/2}\sqrt{4(k-1) - k^2x}{\bf 1}_{[0,4(k-1)/k^2]}(x) dx
	\\& = \frac{2^{2n+2}(k-1)^{n+1}}{2\pi k^{2n+1}} \int x^{n-1/2}\sqrt{1-x}{\bf 1}_{[0,1]}(x) dx
	\\& =  \frac{2^{2n}(k-1)^{n+1}}{\sqrt{\pi} k^{2n+1}} \frac{\Gamma(n+1/2)}{(n+1)!}
	\\& =  \frac{(k-1)^{n+1}}{k^{2n+1}} \frac{(2n)!}{(n+1)!n!} = \frac{(k-1)^{n+1}}{k^{2n+1}} C_n.
\end{align*}
The expression of $m_n(\infty)$ follows. 
\end{proof}

\begin{rem}
	Taking the expectation in \eqref{Expansion1}, we infer that  $m_n(\infty) = m_n/k^{2n-1}$. Consequently, the last relation may be written as 
	\begin{equation*}
		k^2m_n - m_{n+1} = (k-1)^n C_n,
	\end{equation*}
or equivalently,
\begin{equation}\label{difference}
	c_{n,1}-c_{n,2}=\frac{(k-1)^{n-1}}{n+1}\binom{2n}{n}.
\end{equation}
This elementary identity will be used in the proof of Lemma \ref{Lem4} below. 
\end{rem}

Now, set
\begin{equation}\label{KN}
	K_{n,0} :=2(k-1)\left(c_{n,1}+(k-2)\sum_{l=2}^{n}(k-1)^{l-2}c_{n,l}\right), \quad n \geq 1,
\end{equation}
where an empty sum is zero. Then
\begin{lem}\label{Lem4}
For any $n\ge 1$, we have
\begin{equation}
	K_{n,0}=(k-1)^{n}\binom{2n}{n}.
\end{equation}
\end{lem}
\begin{proof}
We proceed by induction: $K_{1,0} = 2(k-1)c_{1,1} = 2(k-1)$. Next, assume the result is valid up to order $n$ and write (we recall that $c_{n,j} =0, j > n$):
\begin{align*}
	K_{n+1,0}=&2(k-1)\left(c_{n+1,1}+(k-2)\sum_{l=2}^{n+1}(k-1)^{l-2}c_{n+1,l}\right)
	\\=&2(k-1)\left((2k-1)c_{n,1}+(k-1)^2c_{n,2}+ (k-2)\sum_{l=2}^{n+1}(k-1)^{l-2}\left[2(k-1)c_{n,l}\right.\right.
	\\&\left.\left.+c_{n,l-1}+(k-1)^2c_{n,l+1}\right]\right)
	\\=&2(k-1)\left((2k-1)c_{n,1}+(k-1)^2c_{n,2}+2(k-2)\sum_{l=2}^{n}(k-1)^{l-1}c_{n,l}\right.
	\\&+\left. (k-2) \sum_{l=1}^{n} (k-1)^{l-1} c_{n,l}+ (k-2)\sum_{l=3}^{n}(k-1)^lc_{n,l+1}\right)
	\\=&2(k-1)\left(3(k-1)c_{n,1}+(k-1)c_{n,2}+4(k-2)\sum_{l=2}^{n}(k-1)^{l-1}c_{n,l}\right)
	\\=&4(k-1)K_{n,0}-2(k-1)^2(	c_{n,1}-c_{n,2}).
\end{align*}
Appealing to the induction hypothesis and to the identity \eqref{difference}, we end up with
\begin{align*}
	K_{n+1,0}=&4(k-1)(k-1)^{n}\binom{2n}{n}-2(k-1)^2 \frac{(k-1)^{n-1}}{n+1}\binom{2n}{n}
	\\=& (k-1)^{n+1}\binom{2n+2}{n+1},
\end{align*}
as desired. 
\end{proof}

We are now ready to prove Theorem \ref{Th2}. 
\begin{proof}[Proof of Theorem \ref{Th2}]
We apply Lemma \ref{Lem1} to $a = T_k$ and $b= U_tT_kU_t^{\star}$ and take the expectation with respect to $\tau$. By the trace property and the fact that $\tau(T_k) = \tau(U_tT_kU_t^{\star}) = 0$, we have 
\begin{equation*}
\tau((ab)^j) = \tau((ba)^j), \quad \tau((ab)^ja)= (k-2)\sum_{l=1}^{j}(k-1)^{j-l}\tau((ab)^l) = \tau((ba)^jb),
\end{equation*}
whence 
\begin{align}\label{Expansion2}
	\tau\big([({\bf 1} +T_k)({\bf 1} + U_tT_kU_t^{\star})]^n\big)=& m_n+\sum_{j=1}^n K_{n,j}\tau((T_kU_tT_kU_t^{\star})^j)
\end{align}
where
\begin{equation}
	K_{n,j}=\begin{cases}
		c_{n,j}+d_{n,j}+(k-2)\sum_{l=j}^{n-1}(k-1)^{l-j}(e_{n,l}+f_{n,l}),& 1\le j\le n-1\\
		1,& j=n
	\end{cases}.
\end{equation}
Equivalently,  Lemma \ref{Lem1} again entails: 
\begin{align*}
	K_{n,j} &= c_{n,j}+c_{n,j+1}+2(k-2)\sum_{l=j}^{n-1}(k-1)^{l-j}c_{n,l+1} \\  
	& = c_{n,j}+c_{n,j+1}+2(k-2)\sum_{l=j+1}^{n}(k-1)^{l-(j+1)}c_{n,l}, \quad 1\le j\le n.
\end{align*}
Appealing further to \eqref{Sys02}, we obtain 
\begin{eqnarray}\label{R02}
K_{n+1,j} & =& (k-1)^2K_{n,j+1} + 2(k-1)K_{n,j} + K_{n,j-1}, \quad 1 \leq j \leq n+1, 
\end{eqnarray}
with the convention $K_{n,j} =0, j > n$ and with $K_{n,0}$ given by \eqref{KN}. Finally, \eqref{R02} is satisfied by the sequence 
\begin{equation*}
(k-1)^{n-j}\binom{2n}{n-j}
\end{equation*}
as readily seen from the identity
\begin{equation*}
\binom{n}{j} + \binom{n}{j+1} = \binom{n+1}{j+1}.
\end{equation*}
Moreover, Lemma \ref{Lem4} and the obvious value $K_{n,n} = 1, n \geq 1,$ show that the boundary conditions coincide, whence we deduce: 
\begin{equation*}
K_{n,j} = (k-1)^{n-j}\binom{2n}{n-j}.
\end{equation*}
Noting that
\begin{equation*}
\tau[(PU_tPU_t^{\star})^n] = \frac{1}{k^{2n}}\tau\big([({\bf 1} +T_k)({\bf 1} + U_tT_kU_t^{\star})]^n\big),
 \end{equation*}
the Theorem is proved. 
\end{proof}

\begin{rem}[Combinatorial approach]
Applying the moment formula with product as entries, it follows that: 
\begin{equation*}
\tau(J_t^n) = \sum_{NC(2n)} \kappa_{\pi}(\underbrace{P, \dots, P}_{2n}) \tau_{K(\pi)} (U_t, U_t^{\star}, \dots, U_t, U_t^{\star})).
\end{equation*}\label{FreeCumPk2}
When $\tau(P) = 1/2$, it is known that the free cumulants of $P$ are given by (\cite{Nic-Spe}, Exercise 11.35): 
\begin{equation}
\tau_{2j+1}(P) = \frac{\delta_{j0}}{2}, \quad \tau_{2j}(P) = \frac{(-1)^{j-1}}{2^{2j}}C_{j-1}.
\end{equation}
It would be interesting to recover \eqref{Expk2} using these formulas together with properties of non crossing partitions. More generally, we can prove (see Appendix B) that if $P$ is a self-adjoint projection with $\tau(P) = \alpha$, then 
\begin{equation}\label{FreeCumP}
\kappa_1(P) = \alpha, \quad \kappa_n(P) = \frac{1}{2(2n-1)} \left[P_{n-2}(1-2\alpha) - P_{n}(1-2\alpha)\right], \quad n \geq 2.
\end{equation}
where $(P_n)_{n \geq 0}$ is the family of Legendre polynomials defined through the Gauss hypergeometric function by: 
\begin{equation*}
P_n(x) = {}_2F_1\left(-n,n+1, 1; \frac{1-x}{2}\right), \quad x \in [-1,1].
\end{equation*} 
\end{rem}

\subsection{PDE for the moment generating function}
Let 
\begin{equation*}
M_{t,k}(z) := k\sum_{n \geq 0} \tau(J_t^n) z^n, \quad \rho_{t,k} :=  \sum_{n \geq 1} \tau[(T_kU_tT_kU_t^{\star})^n]\frac{z^n}{(k-1)^n}, 
\end{equation*}
be the moment generating functions of $J_t$ in the compressed space $(P\mathscr{A}P, k\tau)$ and of $T_kU_tT_kU_t^{\star}$ in $(\mathscr{A}, \tau)$. Both series have positive convergence radii since the corresponding operators are bounded. 
From Theorem \ref{Th2}, we deduce the following relation: 
\begin{cor}
For any $k \geq 2$ and any $t > 0$, 
\begin{equation*}
M_{t,k}(z) = M_{\infty, k}(z) + \frac{k^2}{\sqrt{k^2-4(k-1)z}} \rho_{t,k}\left(\alpha\left[\frac{4(k-1)z}{k^2}\right]\right).
\end{equation*}
where 
\begin{equation}\label{SMGF}
M_{\infty, k}(z) := \sum_{n \geq 0}m_n(\infty) z^n = \frac{2-k + \sqrt{k^2 - 4(k-1)z}}{2(1-z)}, \quad |z| < 1, 
\end{equation}
is the moment generating function of $J_{\infty}$ and 
\begin{equation*}
\alpha(z) = \frac{1-\sqrt{1-z}}{1+\sqrt{1-z}}, \quad z \in \mathbb{C} \setminus [1,\infty[.
\end{equation*}
\end{cor}
\begin{proof}
It is obvious from Theorem \ref{Th2} that 
\begin{equation*}
M_{t,k}(z) = M_{\infty, k}(z) + k \sum_{n \geq 1}\frac{((k-1)z)^n}{k^{2n}} \sum_{j=1}^n\binom{2n}{n-j} \frac{\tau[(T_kU_tT_kU_t^{\star})^j]}{(k-1)^j}. 
\end{equation*}
The expression of $M_{\infty,k}$ is already known (see e.g. section 5 in \cite{Dem}). Now, recall the following result (\cite{Man-Sri}, p.357): if $(a_n)_{n \geq 0}, (b_n)_{n \geq 0}$ are two real sequences satisfying 
\begin{equation*}
b_n = \sum_{j=0}^n\binom{2n}{n-j} a_j,
\end{equation*}
then 
\begin{equation*}
\sum_{n \geq 0}b_n\frac{w^n}{4^n} = \frac{1+\alpha(w)}{1-\alpha(w)} \sum_{n \geq 0} a_n[\alpha(w)]^n
\end{equation*}
whenever both series converge absolutely. Applying this result with 
\begin{equation*}
a_0 = b_0 = 0, \quad a_j = \frac{\tau[(T_kU_tT_kU_t^{\star})^j]}{(k-1)^j},\,  j \geq 1, \quad w = \frac{4(k-1)z}{k^{2}}, 
\end{equation*}
and noting that 
\begin{equation*}
\frac{1+\alpha(w)}{1-\alpha(w)} = \frac{1}{\sqrt{1-w}}
\end{equation*}
conclude the proof. 
\end{proof}

From this corollary, we can derive a pde for $\rho_{t,k}$: 
\begin{pro}
The moment generating function $\rho_{t,k}(z)$ satisfies the pde: 
\begin{equation}\label{pde0}
\partial_t\rho_{t,k}(z) = -z\partial_z\left[\rho_{t,k}(z) + \frac{4(k-1)-k^2\alpha^{-1}(z)}{4(k-1)(1-\alpha^{-1}(z))}\rho_{t,k}^2(z)\right],
\end{equation}
in a neighborhood of the origin with the initial condition :
\begin{equation*}
	\rho_{0,k}(z) = \frac{(k-1)z(1-z)}{(k-1-z)(1+z-kz)}.
\end{equation*}
\end{pro}
\begin{proof}
Let 
\begin{equation*}
G_{t,k}(z) := \frac{1}{z}M_{t,k}\left(\frac{1}{z}\right), \quad |z| > 1,
\end{equation*}
be the Cauchy Stieltjes transform of the free Jacobi process $PU_tPU_t^{\star}P$ with $\tau(P) = 1/k$, and recall from \cite{Dem} that it satisfies the pde: 
\begin{equation*}
\partial_tG_{t,k}(z) = \frac{1}{k}\partial_z[(k-2)zG_{t,k}(z) + z(z-1)G_{t,k}^2(z)]. 
\end{equation*}
Then the variable change $z \mapsto 1/z$ shows that 
\begin{equation}\label{pde1}
\partial_tM_{t,k}(z) = -\frac{z}{k}\partial_z[(k-2)M_{t,k}(z) + (1-z)M_{t,k}^2(z)]
\end{equation}
Now set $R_{t,k} := M_{t,k} - M_{\infty,k}$, then we further get:  
\begin{align*}
\partial_tR_{t,k}(z) & = -\frac{z}{k}\partial_z[(k-2)R_{t,k}(z) + 2(1-z)R_{t,k}(z)M_{\infty,k} + (1-z)R_{t,k}^2(z)] 
\\& =  -\frac{z}{k}\partial_z[\sqrt{k^2-4(k-1)z}R_{t,k}(z) + (1-z)R_{t,k}^2(z)],
\end{align*}
where the first equality follows from the fact that $M_{\infty,k}$ is a stationary solution of the pde \eqref{pde1}: 
\begin{equation*}
\partial_z[(k-2)M_{\infty,k}(z) + (1-z)M_{\infty,k}^2(z)] = 0,
\end{equation*}
and the second one follows from \eqref{SMGF}. Noting that 
\begin{equation*}
\sqrt{k^2-4(k-1)z}R_{t,k}(z) =  k^2\rho_{t,k}\left(\alpha\left[\frac{4(k-1)z}{k^2}\right]\right), 
\end{equation*}
it follows that the map $(t,z) \mapsto \rho_{t,k}(\alpha(z))$ satisfies:
\begin{equation*}
\partial_t\rho_{t,k}(\alpha(z)) = -z\sqrt{1-z}\partial_z\left[\rho_{t,k}(\alpha(z)) + \frac{4(k-1)-k^2z}{4(k-1)(1-z)}\rho_{t,k}^2(\alpha(z))\right].
\end{equation*}
Next, $\alpha$ is a one-to-one holomorphic map in $\mathbb{C}\setminus [1,\infty[$ onto the open unit disc with inverse given by: 
\begin{equation*}
\alpha^{-1}(z) = \frac{4z}{(1+z)^2}. 
\end{equation*}
Moreover
\begin{equation*}
\alpha'(z) = \frac{\alpha(z)}{z\sqrt{1-z}}
\end{equation*}
whence 
\begin{equation*}
[\alpha^{-1}]'(z) = \frac{\alpha^{-1}(z)\sqrt{1-\alpha^{-1}}(z)}{z}.
\end{equation*}
The sought pde satisfied by $\rho_{t,k}(z)$ follows after few computations. 
Finally,
\begin{equation*}
	 \rho_{0,k}(z) = \sum_{n \geq 1} \tau[(T_k)^{2n}]\frac{z^n}{(k-1)^n}.
\end{equation*}
Letting $h_n := \tau[(T_k)^n], n \geq 0,$ then we can easily prove using the relation $T_k^2 = (k-2)T_k + (k-1)$ that 
\begin{equation*}
	h_{n+2} = (k-2)h_{n+1}+ (k-1)h_n, \quad h_0 = 1, h_1 = 0.
\end{equation*} 
This is a generalized Fibonacci sequence for which a Binet formula already exists (\cite{Hor-Koc}):
\begin{equation*}
	h_n=\frac{(k-1)^n+(-1)^n(k-1)}{k}, \quad n\ge0.
\end{equation*}
Then,
\begin{equation*}
	\rho_{0,k}(z) = \sum_{n \geq 1} h_{2n}\frac{z^n}{(k-1)^n}=\frac{(k-1)z(1-z)}{(k-1-z)(1+z-kz)}.
\end{equation*}
\end{proof}

Setting $\eta_{t,k}(z) := \rho_{t,k}(e^{t}z)$ then 
\begin{align*}
\partial_{t}\eta_{t,k} &= \partial_t\rho_{t,k}(e^tz) +  e^{t}z\partial_z\rho_{t,k}(e^tz).
\end{align*}
Which yields,
\begin{align}
\label{pde2}
\partial_{t}\eta_{t,k} & =  -z\partial_z\left[\frac{4(k-1)-k^2\alpha^{-1}(e^tz)}{4(k-1)(1-\alpha^{-1}(e^tz))}\eta_{t,k}^2(z)\right].
\end{align}
In particular, if $k=2$ then 
\begin{align*}
	\partial_t\eta_{t,2}(z) =  -z\partial_z\left[\eta_{t,2}^2(z)\right],
\end{align*}
while 
\begin{equation*}
	\eta_{0,2}(z) = \frac{z}{1-z}.
\end{equation*}
In this case, it is known that $\eta_{t,2}$ is the moment generating function of the free unitary Brownian motion $e^tU_{2t}$ (\cite{DHH}): 
\begin{equation*}
\eta_{t,2}(z) = \sum_{n \geq 1}\frac{z^n}{n}L_{n-1}^{(1)}(2nt)
\end{equation*}
where $L_{n-1}^{(1)}$ is the $(n-1)$-th Laguerre polynomial of parameter one. 
However, it turns out that for $k\ge3$ the computations becomes very complicated due to the high non-linearity of the pde \eqref{pde2}, we then postpone its analysis to a future research work. 

\subsection{Characteristic curves of the pde} 
Denote 
\begin{equation*}
\lambda_{k}(z) := \frac{4(k-1)-k^2\alpha^{-1}(z)}{4(1-\alpha^{-1}(z))}, 
\end{equation*}
so that the pde \eqref{pde0} reads: 
\begin{equation*}
\partial_t\rho_{t,k}(z) = -z\partial_z\left[\rho_{t,k}(z) + \frac{\lambda_k(z)}{(k-1)}\rho_{t,k}^2(z)\right]. 
\end{equation*}
Elementary transformations show that $\tilde{\rho}_{t,k}(z) = [\rho_{t,k}(z)]/(k-1)$ satisfies
\begin{equation*}
\partial_t\tilde{\rho}_{t,k}(z) = -z\partial_z\left[\tilde{\rho}_{t,k}(z) + \lambda_k(z)\tilde{\rho}_{t,k}^2(z)\right]. 
\end{equation*}
Let $z$ be fixed in a neighborhood of the origin. Then a characteristic curve starting at $z$ is locally the unique solution of the Cauchy problem: 
\begin{equation*}
z_k'(t) = z(t)[1+2\lambda_k(z_k(t))f_k(t)], \quad z_k(0) = z,
\end{equation*}
where we set $f_k(t) := \tilde{\rho}_{t,k}(z_k(t))$. Along such curve, it holds that: 
\begin{equation*}
(f_k)'(t) = -z(t)(\lambda_k)'(z_k(t))f_k^2(t), \quad f_k(0) =  \tilde{\rho}_{t,k}(z) = \frac{z(1-z)}{(k-1-z)(1+z-kz)}.
\end{equation*}
Now, set
\begin{equation*}
H(u) := \frac{u+1}{u-1} 
\end{equation*}
and note that $H$ is an involution $(H^{-1} = H)$, $H'(u) = -(H(u)-1)^2/2$ and 
\begin{equation*}
\lambda_k(z) = \frac{1}{4}[k^2 - (k-2)^2H^2(z)].  
\end{equation*}
Then the curve defined by $y_k(t) := H(z_k(t))$ solves locally around $-1$ the Cauchy problem:
\begin{equation*}
y_k'(t)  = \frac{1-y_k^2(t)}{2}\left[1+\frac{k^2 - (k-2)^2y_k^2(t)}{2}f_k(t)\right],  \quad  \quad  y_k(0) = \frac{z+1}{z-1} := y.
\end{equation*}
Besides,
\begin{align*}
(f_k)'(t) &= -\frac{(k-2)^2}{4} y_k(t)H(y_k(t)) (y_k(t) -1)^2f_k^2(t) 
\\& = \frac{(k-2)^2(1-y_k^2(t))}{4}y_k(t)f_k^2(t). 
\end{align*}
Consequently, 
\begin{align*}
y_k'(t)y_k(t)\frac{(k-2)^2(1-y_k^2(t))}{4}f_k^2(t) &= \frac{1-y_k^2(t)}{2}\left[1+\frac{k^2 - (k-2)^2y_k^2(t)}{2}f_k(t)\right](f_k)'(t),
\end{align*}
or equivalently 
\begin{align*}
\frac{(k-2)^2}{4}\left[(y_k^2)'(t)f_k^2(t) + y_k^2(t)(f_k^2)'(t)\right] &= \left[1+\frac{k^2}{2}f_k(t)\right](f_k)'(t). 
\end{align*}
This equation is integrable and yields: 
\begin{align*}
\frac{(k-2)^2}{4}\left[(y_k^2)(t)f_k^2(t) - y_k^2(0)(f_k^2)(0)\right] &= \left[\frac{k^2}{4}f_k^2(t) + f_k(t) - \frac{k^2}{4}f_k^2(0) - f_k(0)\right].
\end{align*}
Written differently leads to the functional equation: 
\begin{equation*}
\tilde{\lambda}_k(y(t))f_k^2(t) + f_k(t) - \tilde{\lambda}_k(y)f_k^2(0) - f_k(0) = 0,
\end{equation*}
where we simply wrote
\begin{equation*}
\tilde{\lambda}_k(y) = \frac{1}{4}[k^2 - (k-2)^2y^2] = \lambda_k(z).
\end{equation*}
Setting $g_k(0) := \tilde{\lambda}_k(y)f_k^2(0) + f_k(0)$, then one has locally:
\begin{equation*}
f_k(t) = \frac{-1+\sqrt{1+4g_k(0)\tilde{\lambda}_k(y_k(t))}}{2\tilde{\lambda}_k(y_k(t))},
\end{equation*}
where the principal determination of the square root is considered. It follows that:
\begin{align*}
y_k'(t) &= \sqrt{1+4g_k(0)\tilde{\lambda}_k(y_k(t))}\frac{1-y_k^2(t)}{2} 
\\& = \frac{1-y_k^2(t)}{2}\sqrt{1+k^2g_k(0) - (k-2)^2g_k(0)(y_k^2(t))}.
\end{align*}
Now, consider the indefinite integral 
\begin{equation*}
I_{A,B}(u) = 2\int^u \frac{du}{(1-u^2)\sqrt{A -Bu^2}}
\end{equation*}
for two indeterminates $(A,B)$ independent of the variable $u$. Then 
\begin{equation*}
I_{A,B}(u) = \frac{1}{\sqrt{A-B}}\log\left[\frac{(\sqrt{A-Bu^2} + \sqrt{A-B}u)^2}{A(1-u^2)}\right],
\end{equation*}
provided the square root is well-defined (we can take any determination of the logarithm). Taking $A = 1+ k^2g_k(0), B = (k-2)^2g_k(0)$, one gets
\begin{equation*}
I_{A,B}(y(t)) - I_{A,B}(y)  = t,
\end{equation*}
or after exponentiating this identity:
\begin{align*}
\frac{(\sqrt{A-B(y_k(t))^2} + \sqrt{A-B}y_k(t))^2}{A(1-y_k(t)^2)} = e^{\sqrt{A-B}t} \frac{(\sqrt{A-By^2} + \sqrt{A-B}y)^2}{A(1-y^2)}
\end{align*}
Noting that 
\begin{equation*}
f_k(0) = -\frac{1+y}{2\tilde{\lambda}_k(y)} \quad 4g_k(0) = \frac{y^2-1}{\tilde{\lambda}_k(y)}, 
\end{equation*}
then $A-By^2 = 1 + 4g_k(0) \tilde{\lambda}_k(y) = y^2$ which in turn entails: 
\begin{align*}
\frac{(\sqrt{A-B(y_k(t))^2} + \sqrt{A-B}y_k(t))^2}{A(1-y_k(t)^2)} & = e^{\sqrt{A-B}t} \frac{y^2(1 - \sqrt{A-B})^2}{A(1-y^2)} 
\\& = e^{\sqrt{A-B}t} \frac{1 - \sqrt{A-B}}{1+\sqrt{A-B}}.
\end{align*}
If we denote the LHS of the second equality $\xi_{2t}(\sqrt{A-B})$, then lengthy computations yield: 
\begin{align*}
y_k^2(t) &= \frac{A(1-\xi_{2t}(\sqrt{A-B}))^2}{A(1+\xi_{2t}(\sqrt{A-B}))^2 -4B\xi_{2t}(\sqrt{A-B})} 
\\& =  \frac{A}{(A-B)H^2[\xi_{2t}(\sqrt{A-B})] +B} 
\\& = \frac{1+k^2g_k(0)}{(1+4(k-1)g_k(0))[H^2[\xi_{2t}(\sqrt{1+4(k-1)g_k(0)}]-1] + (1+k^2g_k(0))}. 
\end{align*}
The map $(-\xi_{2t})$ is the inverse of the Herglotz transform $1+2\eta_{t,2}$ of the spectral distribution of $U_{2t}$ in a neighborhood of $u=1$. As a matter of fact, the map 
\begin{equation*}
u \mapsto \frac{1+k^2u}{(1+4(k-1)u)[H^2[\xi_{2t}(\sqrt{1+4(k-1)u}]-1] + (1+k^2u)}
\end{equation*}
is a locally invertible in a neighborhood of the origin $u=0$. Let $\zeta_{2t}$ be its inverse then 
\begin{equation*}
g_k(0) = \zeta_{2t}(y_k^2(t)),
\end{equation*}
and in turn 
\begin{align*}
f_k(t)  &= \frac{-1+\sqrt{1+4\zeta_{2t}(y_k^2(t)) \tilde{\lambda}_k(y_k(t))}}{2\tilde{\lambda}_k(y_k(t))} 
\\& = 2 \frac{\zeta_{2t}[H^2(z_k(t))]}{1+\sqrt{1+\zeta_{2t}[H^2(z_k(t))][k^2-(k-2)^2H^2(z_k(t))]}} =  \tilde{\rho}_{t,k}(z_k(t)).
\end{align*}
\begin{rem}
If $k=2$ then 
\begin{equation*}
y_2^2(t) = \frac{1}{H^2[\xi_{2t}(\sqrt{1+4g_2(0)}]} = \frac{1}{H^2[\xi_{2t}(-y)]} = H^2[-\xi_{2t}(-y)].
\end{equation*}
Thus, it holds locally:
\begin{equation*}
y_2(t) = H[-\xi_{2t}(-y)] \quad \Rightarrow \quad z_2(t) = (-\xi_{2t})\left(\frac{1+z}{1-z}\right) = ze^{t(1+z)/(1-z)}.  
\end{equation*}
For fixed $t > 0$, the map $z \mapsto z_2(t)$ is known as the $\Sigma$-transform of the spectral distribution of $U_{2t}$ (\cite{Bia}). 
 \end{rem}

\section{Concluding remarks}
So far, we introduced a dynamical random density matrix by means of $k \geq 2$ independent unitary Brownian motions whose large size limit has, up to a normalization, the same moments as those of the free Jacobi process $PU_tPU_t^{\star}P$  (in the compressed algebra) subject to $\tau(P) = 1/k$. Motivated by our previous results proved in \cite{DHH} valid for $k=2$, we derived for any $k \geq 2$ a binomial-type expansion for these moments and gave rise to a non normal (except when $k=2$) operator. In \cite{Ham}, another approach is undertaken and relies rather on the spectral dynamics of the unitary operator $T_2U_tT_2U_t^{\star}$. Specializing Theorem 1.1. there to $P=Q$ with $\tau(P) = 1/k$, the spectral distribution of $PU_tPU_t^\star P$ 
admits a density given by:
\begin{equation*}
\mu_t^k(x)=\frac{1}{2\pi}\frac{g_t^k\left(2\arccos(\sqrt{x})\right)}{\sqrt{x-x^2}},
\end{equation*}
where $g^k_t$ is the density of the spectral distribution of $T_2U_tT_2U_t^{\star}$. However, for $k \geq 3$, $g_t^k$ admits a very complicated expression compared to the simple one corresponding to $k=2$. 

On the other hand, it was recently proved in \cite{Dem-Ham2} (see eq. (3) there) that 
\begin{equation*}
\tau[(PQP)^{n}] - \tau[(({\bf 1} - P)({\bf 1} - Q)({\bf 1}-P))^n] = \tau(P) + \tau(Q) - 1,
\end{equation*}
for any self-adjoint projections $(P,Q)$ with arbitrary traces $\tau(P), \tau(Q) \in (0,1)$. In particular, if $Q = U_tPU_t^{\star}$ and $\tau(P) = 1/k$ then we readily deduce: 
\begin{equation*}
\tau[(({\bf 1} - P)U_t({\bf 1} - P)U_t^{\star}({\bf 1} - P))^n] = \tau[(PU_tPU_t^{\star}P)^n] + 1 -\frac{2}{k}.
\end{equation*} 
Equivalently, the moments of the free Jacobi process associated with ${\bf 1}-P$ may be deduced from those of the free Jacobi process associated with $P$: 
\begin{equation*}
\frac{k}{k-1}\tau[(({\bf 1} - P)U_t({\bf 1} - P)U_t^{\star}({\bf 1} - P))^n] = \frac{k}{k-1}\tau[(PU_tPU_t^{\star}P)^n] + \frac{k-2}{k-1}.
\end{equation*} 
This fact reminds the duality between linear subspaces and their complementaries in Grassmann manifolds and is not surprising since the free Jacobi process is the large size limit of the radial part of the Brownian motion in the complex Grassmann manifold (see \cite{Dem} for further details). 

\appendix 
\section{Moments of stationary distribution}
In this appendix, we derive another expression of 
\begin{equation*}
m_n(\infty) = \int x^n \tilde{\mu}_{\infty}^k(dx) = \frac{1}{2\pi} \int x^{n-1/2}\frac{\sqrt{4(k-1) - k^2x}}{(1 -x)}{\bf 1}_{[0,4(k-1)/k^2]}(x) dx.
\end{equation*}
To the best of our best knowledge, formula \eqref{MomStat2} below never appeared in literature. Compared to \eqref{MomStat2}, it has the merit to separate the case $k=2$ corresponding to the arcsine distribution from other values $k \geq 3$. 
Our main ingredients are two properties satisfied by the Gauss hypergeometric function. 

To proceed, perform the variable change $x = 4(k-1)y/k^2$ to write:
\begin{align*}
m_n(\infty) &= \frac{4^{n+1}(k-1)^{n+1}}{2\pi k^{2n+1}}\int y^{n-1/2}\frac{\sqrt{1-y}}{1 - 4(k-1)y/k^2} {\bf 1}_{[0,1]}(y) dy 
\\&= \frac{4^{n}(k-1)^{n+1}}{\sqrt{\pi} k^{2n+1}}\frac{\Gamma(n+1/2)}{\Gamma(n+2)} {}_2F_1\left(1, n+\frac{1}{2}, n+2; \frac{4(k-1)}{k^2}\right).
\\& =  \frac{4^{n}(k-1)^{n+1}}{n!k^{2n+1}}\left\{(-1)^n\sqrt{1-z}\frac{d^n}{dz^n} \left[(1-z)^{n-1/2}{}_2F_1\left(\frac{1}{2},1,2; z\right)\right]\right\}_{z = 4(k-1)/k^2} 
 \\& = 2\frac{4^{n}(k-1)^{n+1}}{n!k^{2n+1}}\left\{\sqrt{1-z}(-1)^n\frac{d^n}{dz^n} \left[\frac{(1-z)^{n-1/2}}{1+\sqrt{1-z}}\right]\right\}_{z = 4(k-1)/k^2} 
 \\& = 2\frac{4^{n}(k-1)^{n+1}}{n!k^{2n+1}}\left\{\sqrt{z}\frac{d^n}{dz^n} \left[\frac{z^{n-1/2}}{1+\sqrt{z}}\right]\right\}_{z = [(k-2)/k]^2} 
 \\& = 2\frac{4^{n}(k-1)^{n+1}}{n!k^{2n+1}}\left\{\sqrt{z}\frac{d^n}{dz^n} \left[z^{n-1/2} - \frac{z^n}{1+\sqrt{z}} \right] \right\}_{z = [(k-2)/k]^2} 
 \\& = \frac{2(k-1)^{n+1}}{k^{2n+1}}\left\{\binom{2n}{n} -\frac{4^{n}}{n!}\sqrt{z}\frac{d^n}{dz^n}\left[\frac{z^n}{1+\sqrt{z}} \right] \right\}_{z = [(k-2)/k]^2}.
 \end{align*}
 Here ${}_2F_1$ is the hypergeometric function, the second equality follows from its Euler integral representation, the third and fourth ones follow from the variational formula (25), p.102 and formula (6), p.101 in Erdelyi's book. 
 Using direct computations, we readily see that 
 \begin{equation*}
\frac{d^n}{dz^n}\left[\frac{z^n}{1+\sqrt{z}} \right] = \frac{\mathscr{P}_n(\sqrt{z})}{2^n(1+\sqrt{z})^{n+1}}
 \end{equation*}
 for some polynomial of degree $n$. For instance 
 \begin{equation*}
 \mathscr{P}_0(x) = 1, \mathscr{P}_1(x) = x+2, \mathscr{P}_2(x) = 3x^2+9x+8, \mathscr{P}_3(x) = 15x^3+60x^2+87x+48. 
 \end{equation*}
 Consequently, 
 \begin{equation}\label{MomStat2}
m_n(\infty) = \frac{2(k-1)^{n+1}}{k^{2n+1}}\left\{\binom{2n}{n} -\frac{(k-2)}{2n!(k-1)^{n+1}}k^n\mathscr{P}_n\left(\frac{k-2}{k}\right)\right\}. 
\end{equation}

\section{Free cumulants of an orthogonal projection}
The first part of the proof is a routine computation in free probability theory and we refer the reader to \cite{HP} for further details on this machinery. 
Start with the Cauchy transform of $P$: 
\begin{equation*}
\tau[(z-P)^{-1}] = \frac{\alpha}{z-1} + \frac{1-\alpha}{z} = \frac{z + \alpha-1}{z(z-1)}, \quad z \notin \{0,1\}.
\end{equation*}
Next, consider the equation 
\begin{equation*}
yz^2 -z(y+1) + 1-\alpha = 0, 
\end{equation*}
for $y$ lying in a neighborhood of zero. Then the $K$-transform of $P$ reads: 
\begin{equation*}
K(y) = \frac{y+1 + \sqrt{y^2+1 -2y(1-2\alpha)}}{2y} 
\end{equation*}
and in turn, its $R$-transform is given by 
\begin{equation*}
R(y) =  K(y) -\frac{1}{y} = \frac{1}{2}\left[1 + \frac{\sqrt{y^2+1 -2y(1-2\alpha)}-1}{y}\right]
\end{equation*}
It remains to write down the Taylor expansion of the function: 
\begin{equation*}
f_{\alpha}: y \mapsto \frac{\sqrt{y^2+1 -2y(1-2\alpha)}-1}{y}. 
\end{equation*}
To this end, we appeal to the generating series of Legendre polynomials: 
\begin{equation*}
\sum_{n=0}^{\infty} P_n(x)y^n = \frac{1}{\sqrt{1+y^2-2xy}}, \quad |y| < 1. 
\end{equation*}
Indeed, setting $\beta := 1-2\alpha \in [-1,1]$, one has 
\begin{equation*}
[yf_{\alpha}(y)]' = (y - \beta) \sum_{n\geq 0} P_n(\beta) y^n 
\end{equation*}
so that 
\begin{align*}
f_{\alpha}(y) &= \sum_{n \geq 1}\frac{y^{n}}{n+1} [P_{n-1}(\beta) - \beta P_n(\beta)] - \beta 
\\& = \sum_{n \geq 1}\frac{y^{n}}{2n+1} \left[P_{n-1}(\beta) - P_{n+1}(\beta)\right] - \beta
\end{align*}
where the last equality follows from the recurrence relation: 
\begin{equation*}
(2n+1)xP_n(x) = (n+1)P_{n+1}(x) + nP_{n-1}(x).
\end{equation*}
Extracting the Taylor coefficients of $f_{\alpha}$ and recalling the definition
\begin{equation*}
R(y) =\sum_{n \geq 0}\kappa_{n+1}(P)y^n
\end{equation*}
 we get \eqref{FreeCumP}. 
 
Note that since Legendre polynomials are orthogonal with respect to the uniform distribution in $[-1,1]$, they are parity preserving. In particular, 
\begin{equation*}
P_{2n+1}(0) = 0, \quad P_{2n}(0) = (-1)^n\frac{(1/2)_n}{n!},
\end{equation*}
so that one recovers \eqref{FreeCumPk2} after some computations. 



 \end{document}